\documentclass[11pt]{article}
\usepackage[T1]{fontenc}
\usepackage{lmodern}
\usepackage[margin=1in]{geometry}
\usepackage{amsmath,amssymb,amsthm,mathtools}
\usepackage{microtype}
\usepackage[colorlinks=true,linkcolor=blue,citecolor=blue,urlcolor=blue]{hyperref}
\usepackage[
    backend=biber,
    style=numeric,
    sorting=nyt,
    doi=false,
    url=false,
    eprint=true,
    maxbibnames=99
]{biblatex}
\DeclareFieldFormat[article,misc,inproceedings,incollection,unpublished]{title}{#1\isdot}
\newtheorem{theorem}{Theorem}[section]
\newtheorem{proposition}[theorem]{Proposition}
\newtheorem{lemma}[theorem]{Lemma}

\theoremstyle{remark}

\DeclareMathOperator{\Var}{Var}
\newcommand{\E}{\mathbb E}
\newcommand{\R}{\mathbb R}
\newcommand{\Prob}{\mathbb P}
\newcommand{\ind}{\mathbf 1}
\numberwithin{equation}{section}

\title{\huge Unimodality of Independence Polynomials 
for \\Sufficiently Large Forests}
\author{
	Ethan X. Fang\thanks{Department of Biostatistics \& Bioinformatics, Duke
	University, Durham, NC 27710, USA. Email: \texttt{ethan.fang@duke.edu}.}
	\quad
    Junwei Lu\thanks{Department of Biostatistics, Harvard T.H. Chan School of
	Public Health, Boston, MA 02115, USA. Email: \texttt{junweilu@hsph.harvard.edu}.}
    \quad
     Eran Nevo\thanks{Einstein Institute of Mathematics, Hebrew University, Jerusalem 91904, Israel and Institute of Mathematics, Universidad de Valladolid, Valladolid 47011, Spain. Email: \texttt{nevo@math.huji.ac.il}.}
     \quad 
	Yuan Yao\thanks{Department of Biostatistics, Harvard T.H. Chan School of
	Public Health, Boston, MA 02115, USA. Email: \texttt{yyao1@hsph.harvard.edu}.}
	\quad
	Hailun Zheng\thanks{Department of Mathematics, University of Hawai`i at
	M\={a}noa, 2565 McCarthy Mall, Honolulu, HI 96822, USA. Email:
	\texttt{hailunz@hawaii.edu}}
}
\date{}
\begin{document}
\maketitle

\begin{abstract}
We prove that the independence sequence of every sufficiently large
forest is unimodal. 
The result follows from establishing log-concavity on
a central interval of the sequence, along with
monotonicity of the initial and final segments.
The main analytic step is a central limit theorem for the size of a
random independent set sampled with the hard-core model, uniform over all forests and over an interval
of positive fugacities. 
\end{abstract}

\section{Introduction}

For a finite simple graph \(G\), an \emph{independent set}
is a subset of vertices, no two of which are connected by an edge.
Let \(\mathcal I(G)\) be the family
of all independent sets in $G$, let \(\alpha(G)\) be the maximum size of an independent set in $G$, and
write
\[
 Z_G(x)=\sum_{I\in\mathcal I(G)}x^{|I|}
       =\sum_{k=0}^{\alpha(G)}i_k(G)x^k
\]
for its \emph{independence polynomial}, 
where $i_k(G)$ is the number of independent sets of size $k$ in $G$.
The sequence of coefficients
$(i_0(G),\ldots,i_{\alpha(G)}(G))$ is \emph{unimodal} if it is weakly
increasing up to some index and weakly decreasing thereafter. It is
\emph{log-concave} if
\(i_k(G)^2\ge i_{k-1}(G)i_{k+1}(G)\) at each index $1<k<\alpha(G)$.
Since these coefficients are positive, log-concavity implies
unimodality.

Alavi, Malde, Schwenk, and Erd\H{o}s~\cite{AMSE} showed that independence
sequences of 
graphs can have 
arbitrary increasing and decreasing behavior,
and asked whether the sequence is nevertheless unimodal for every
tree and every forest. Note that the forest version of this problem  
does not follow from 
the tree version: the independence polynomial of a graph is the product of 
independence polynomials
over its 
connected components, but unimodality is not preserved under convolution.

In this paper, we resolve the unimodality conjecture
for all sufficiently large forests.

\begin{theorem}\label{thm:main}
There is an absolute positive integer \(N_0\) such that the independence
sequence of every forest with at least \(N_0\) vertices is unimodal.
\end{theorem}

Prior to this work, unimodality of independence polynomials has been established for many specific families of graphs:
Hamidoune~\cite{Hamidoune} showed that all claw-free graphs have unimodal independence sequences.
Zhu~\cite{Zhu}, Wang and Zhu~\cite{WZ}, Galvin and Hilyard~\cite{GH}, and Bahls, Ethridge and Szabo~\cite{BES} proved unimodality for caterpillar graphs and several variants.
More recently, Li, Li, Yang, and Zhang~\cite{LLYZ, Li} proved unimodality for spider graphs and a few other parameterized families by using chromatic symmetric functions. The independence polynomials of some of these families are shown to be even log-concave. It is verified computationally that all trees with at most $25$ vertices always have log-concave independence sequences; this is communicated by Radcliff in~\cite{BGHW}.

There are also results that establish monotonicity for the two ends of the independence sequence. 
Levit and Mandrescu~\cite{LM} proved that the independence sequence of a K\"onig--Egerv\'ary graph is decreasing from index \(\lceil(2\alpha-1)/3\rceil\) onward; in particular, this property holds for every forest.
Basit and Galvin~\cite{BG} strengthened the decreasing-range result
for general graphs and obtained longer increasing and decreasing
ranges for random trees. 

Levit and Mandrescu~\cite{LM2} further conjectured that the independence sequence is always log-concave for forests.
However, Kadrawi, Levit, Yosef, and Mizrachi~\cite{KL, KLYM} found the first
\(26\)-vertex counterexamples to log-concavity and constructed
several infinite families of counterexamples.
Ramos and Sun~\cite{RS} found a large number of counterexamples by
machine-learning-assisted search, while Bautista-Ramos, Guill\'en-Galv\'an, and G\'omez-Salgado~\cite{BGG} constructed
families with log-concavity failures at multiple consecutive indices. In all these counterexamples, the non-log-concave indices 
are concentrated towards the top end of the sequence (near $\alpha(G)$), 
although Galvin~\cite{Galvin} showed that such failures need not
remain within a bounded distance of the final coefficient. 

Nevertheless, we show that log-concavity holds
on a central interval:
\begin{theorem}\label{thm:central}
For every sufficiently large integer \(n\), every forest \(F\) on
\(n\) vertices satisfies
\[
 i_k(F)^2>i_{k-1}(F)i_{k+1}(F)
 \qquad\text{if}\qquad
 \frac n5
 \le k\le
 \frac{17\alpha(F)}{25}.
\]
\end{theorem}
The constants in this theorem are not optimal; in particular, $n/5$ can be replaced by $\varepsilon n$ for any fixed $\varepsilon >0$ at a cost of increasing the lower bound on $n$. (See Section~\ref{sec:Conclude} for further discussions.) The constants here are chosen to allow us to conclude Theorem~\ref{thm:main} when combined with monotonicity of 
suitable initial and terminal intervals 
of the independence sequence.

The proof of Theorem~\ref{thm:central} uses a probabilistic framework based on the hard-core model. For a given fugacity \(\lambda>0\),
an independent set \(I\) is sampled with probability proportional to
\(\lambda^{|I|}\). We study the probability that the sampled independent set has size $k$
to obtain information about the coefficients of the independence polynomial. The proof follows a central-to-local limit approach: we first establish a central
limit theorem uniformly over all forests, then refine this global picture of the 
independent set 
size distribution to compare neighboring coefficients 
in the independence sequence 
and prove log-concavity in a central range of its indices. This approach builds on a long tradition of using normal and local limit
theorems in asymptotic enumeration; see Bender~\cite{Bender} for early
examples. In the setting of Gibbs measures, Dobrushin and Tirozzi~\cite{DT} developed a method of passing from central to local limit theorems using conditional independence. More recently, Jain,
Perkins, Sah, and Sawhney~\cite{JPSS} used a similar strategy to efficiently estimate $i_k(G)$ for bounded-degree graphs, combining
normal approximation, characteristic-function estimates, and
appropriate choices of fugacity.
In this paper, we bypass the bounded-degree
condition 
needed in~\cite{JPSS} 
via moment bounds (see Proposition~\ref{prop:root-moments})
derived directly from the standard
recursion on a rooted tree, same as the one underlying the work of
Weitz~\cite{Weitz} in computing the probability that a given vertex is included in the independent set. 

The paper is organized as follows: We first sketch the overall proof in Section~\ref{sec:sketch}, then set up the hard-core model in Section~\ref{sec:hardcore}. In Section~\ref{sec:root} we bound the impact of revealing a vertex to the mean and variance. In Section~\ref{sec:clt} we use these bounds to show a central limit theorem, which is then upgraded 
in Section~\ref{sec:fourier} 
to a bound on the characteristic function, 
which in turn is used in 
Section~\ref{sec:mean} we prove the central-interval log-concavity claimed in Theorem~\ref{thm:central}. In Section~\ref{sec:ends} we prove the monotonicity of the end segments needed for Theorem~\ref{thm:main}. In Section~\ref{sec:Conclude} we end with some concluding remarks.

\vspace{-15pt}

\paragraph{The role of AI.} The Odin Automatic AI Research Agent was used to find the initial proof. The final content and proofs in this paper were written and verified by the authors. The Lean code formalizing the proof of Theorem~\ref{thm:main} is in the GitHub repository \url{https://github.com/junwei-lu/Erdos_993_Tree_Independent_Set_Unimodality}.

\vspace{-15pt}

\paragraph{Acknowledgements.} We would like to thank Jes\'{u}s A.\ De Loera for his useful suggestions for an earlier draft of this paper.

\section{Proof sketch}\label{sec:sketch}

Let \(X\) be the size of a random independent set 
sampled under the hard-core model on a forest $F$ with \(n\)
vertices, where a set $I$ is sampled with probability proportional to $\lambda^{|I|}$, and denote the mean and variance of $X$ by \(\mu\) and \(\sigma^2\) resp.
The fugacity parameter $\lambda$ will lie in the fixed interval \([1/4,12]\).

\vspace{-15pt}

\paragraph{Normal approximation.}
Conditioning on a root's occupation status
splits a (rooted) tree into smaller trees, but may substantially change the expected 
size of an independent set.
We prove that the square of this change, multiplied by the variance
of the root indicator, is \(O(n^a)\) on an $n$-vertex rooted tree, for a fixed \(a<1\).
We also bound the square of the difference of conditional variances,
multiplied by the variance of the root indicator, by \(O(n^{2a})\); see Eq. (\ref{eq:root-moments}).

We then use a variant of the classic centroid decomposition on each connected component of a forest. 
We reveal the occupation of a centroid, 
delete it or its neighborhood based on the occupation, 
and repeat in the remaining
components until their sizes are at most
\(\lceil n^{1/4}\rceil\).
Component sizes decrease exponentially along every chain of this decomposition, 
allowing us to show that 
the terminal conditional mean differs from \(\mu\) by
\(o(\sqrt n)\) in \(L^2\), and the conditional variance differs
from \(\sigma^2\) by \(o(n)\) in \(L^1\).
Conditioning
on the reveals, \(X\) is
a sum of independent bounded random variables and a known integer.
A Berry--Esseen bound, followed by a comparison of normal
distributions, proves a uniform central limit theorem; see Eq. (\ref{eq:clt}).

\vspace{-15pt}

\paragraph{Bounding characteristic function.}
Weak convergence alone gives no sign information about three
adjacent coefficients. We therefore bound the characteristic
function at all frequencies. Conditional on one side of a
bipartition, available vertices on the other side are independent
Bernoulli variables. Integrating a marginal-occupation inequality
as the fugacity $\lambda$ on that side varies gives a uniform upper bound, see Eq. (\ref{eq:fourier}):
\[
 \bigl|\E_{\lambda} e^{itX}\bigr|\le \exp\{-c n\sin^2(t/2)\}.
\]
Together with \(\sigma^2\asymp n\), this is a Gaussian bound after
standardization. Fourier inversion and the central limit theorem
then imply, whenever the fugacity is chosen so that \(\mu=k\)
is an integer,
\[
 \sigma^3\bigl(2\Prob(X=k)-\Prob(X=k-1)-\Prob(X=k+1)\bigr)
 \longrightarrow \frac1{\sqrt{2\pi}}
\]
uniformly 
over all forests with $n$ vertices and all fugacities in the fixed interval, as $n\rightarrow \infty$; see Eq. (\ref{eq:curvature-limit}).
The positive limit gives strict log-concavity of probabilities at \(k\),
which translates to log-concavity $i_k(F)>i_{k-1}(F)i_{k+1}(F)$ in the claimed central interval of the independence sequence.

\vspace{-15pt}

\paragraph{Bounding mean and monotone intervals.}
The mean of $X$ increases continuously with the fugacity $\lambda$. At
\(\lambda = 1/4\), the mean is at most \(n/5\). At \(\lambda = 12\), a comparison of
\(\log Z_F\), the mean, and the variance shows that the mean exceeds
\(17\alpha(F)/25\).
Thus the preceding argument proves Theorem~\ref{thm:central}.

To finish, we prove that the independence sequence of any forest is increasing
through \(\lceil n/4\rceil\), using a bound on the average number of incident edges to $I$ and extension double counting. 
The sequence is decreasing starting at \(\lceil(2\alpha(F)-1)/3\rceil\), by the classical
result of Levit and Mandrescu, for which we include a short
proof for the bipartite case. For sufficiently large \(n\), the log-concave
interval overlaps both 
of these monotone intervals, 
so patching the three intervals together completes the proof of Theorem~\ref{thm:main}.

\section{The hard-core model setup}\label{sec:hardcore}

Let $G$ be a finite simple graph. We write \(d_G(v)\) for the degree
of \(v\) in \(G\).
For a vertex set \(A\), let \(N(A)\) be its open neighborhood and
\(N[A]=A\cup N(A)\); in particular, \(N[v]=N(\{v\})\cup\{v\}\).
Deleting a vertex set always means taking the induced graph on
its complement. 

For a given \emph{fugacity parameter} \(\lambda>0\), sample an independent set $I$ 
from the collection $\mathcal{I}(G)$ of all independent sets 
according to the distribution
\begin{equation}\label{eq:measure}
 \Prob_\lambda(I=J)=\frac{\lambda^{|J|}}{Z_G(\lambda)}
 \quad(J\in\mathcal I(G)),
 \quad \text{where } Z_G(\lambda):= \sum_{J\in \mathcal I(G)} \lambda^{|J|}.
\end{equation}
A vertex is \emph{occupied} if it is in $I$, and \emph{absent} otherwise. We also use 
\[
 X=|I|,\quad \mu=\E_\lambda X,\quad \sigma^2=\Var_\lambda X
\]
to denote the size of this independent set, as well as the mean and variance of the size.
Subscripts may be omitted when the graph and fugacity are fixed.

Let \(D=\lambda\frac{d}{d\lambda}\). Differentiation of the 
partition function gives
\begin{equation}\label{eq:derivatives}
 D\log Z_G=\mu,\qquad D\mu=\sigma^2.
\end{equation}
For a nonempty graph, \(\sigma^2>0\), since both the empty set and
a singleton have positive probability.

If a vertex is absent, its conditional law on the other vertices
is the hard-core model with that vertex deleted. If it is occupied,
its closed neighborhood is deleted and its contribution to \(X\)
is \(1\). Distinct remaining components are independent. Also,
at fugacity \(\lambda\) the occupation probability for each vertex is at most
\(\lambda/(1+\lambda)\): deleting an occupied vertex injects each set into one in which the vertex is absent, with weight
ratio \(\lambda\). The same bound applies after any specified
vertices have been conditioned to be absent. In particular,
\begin{equation}\label{eq:absence}
 \Prob_\lambda\bigl(I\cap A=\varnothing\bigr)
 \ge (1+\lambda)^{-|A|}
\end{equation}
for every set of vertices \(A\), by conditioning on their absence
one at a time.

For the rest of this paper, we assume that the fugacity $\lambda$ is within the fixed interval
\[
 K=[1/4,12].
\]
The choice of constants is mostly for convenience and not necessarily optimal. In fact, the following proofs go through with the lower bound $1/4$ replaced by any constant $\varepsilon \in (0, 1/3)$.

\section{Root conditioning and changes in the first two moments}\label{sec:root}

For a rooted tree $T$, let \(T_v\) be the descendant subtree at vertex \(v\).
Define the \emph{occupation odds} and \emph{occupation probability} of its root
\emph{in \(T_v\)} (i.e.\ we sample a random independent set within $T_v$ only) by
\[
 r_v=\frac{\mathbb{P}_\lambda(v\in I)}{\mathbb{P}_\lambda(v\notin I)}=\frac{\lambda Z_{T_v-N[v]}(\lambda)}
                 {Z_{T_v-v}(\lambda)},\qquad
 q_v=\mathbb{P}_\lambda(v\in I)=\frac{r_v}{1+r_v},
\]
where $N[v]$ is taken in \(T_v\).
Root conditioning gives the usual recursion
\begin{equation}\label{eq:root-recursion}
 r_v=\lambda\prod_i(1+r_i)^{-1},\qquad
 y_v:=\log(\lambda/r_v)=\sum_i\log(1+r_i),
\end{equation}
where \(i\) ranges over the children of \(v\). In particular
\(0<r_v\le\lambda\) and \(y_v\ge0\).
Let \[
\delta_v = \E_\lambda (X\mid v\in I) - \E_\lambda (X\mid v\notin I) 
= 1 + \mu_{T_v-N[v]} - \mu_{T_v-v},\quad
\gamma_v = \Var_\lambda (X\mid v\in I) - \Var_\lambda (X\mid v\notin I)
\]
be the occupied-root mean minus the absent-root
mean of the total size in \(T_v\), and the
corresponding difference of variances respectively.
The definitions and
\eqref{eq:derivatives} show that
\(\delta_v=D\log r_v\) and \(\gamma_v=D^2\log r_v\).
Applying them to the recurrence on $r_v$ gives
\begin{equation}\label{eq:moment-recursion}
 \delta_v=1-\sum_i q_i\delta_i,\qquad
 \gamma_v=-\sum_iq_i\gamma_i
           -\sum_iq_i(1-q_i)\delta_i^2.
\end{equation}
At a leaf, \(\delta_v=1\) and \(\gamma_v=0\).

\begin{proposition}\label{prop:root-moments}
There are constants \(a\in(2/3,1)\) and \(C \in (0, \infty)\) such that for every
rooted tree $T_v$ with \(n_v\) vertices and every \(\lambda\in K\), the root quantities \(q_v,\delta_v,\gamma_v\) satisfy
\begin{equation}\label{eq:root-moments}
 q_v(1-q_v)\delta_v^2\le C n_v^a,\qquad
 q_v(1-q_v)\gamma_v^2\le C n_v^{2a}.
\end{equation}
\end{proposition}

\begin{proof}
We first bound the difference of means, then use that estimate
in the variance recurrence.
The bounds depend on the following technical lemma:

\begin{lemma}\label{lem:scalar-estimate}
There exists constants $\rho\in (0, 1)$ and $p\in (3/2, 2)$ so that for any \(\lambda\in K\), \(y\ge0\), \(r=\lambda e^{-y}\) and \(q=r/(1+r)\), we have
\begin{equation}\label{eq:scalar-p}
 \frac{yq^p}{\log(1+r)}\le
 \rho.
\end{equation}
\end{lemma}

\begin{proof}
Choose \(b\in(0,1)\) close to \(1\) that
\(12e^{-1-3b/2}<b\). (This is possible because
\(12<e^{5/2}\).) 
We first show that at exponent $2$ there is
\begin{equation}
 \frac{yq^2}{\log(1+r)}\le\frac{yr}{1+3r/2} \le b.
\end{equation}
The first inequality holds since \(\log(1+r)=-\log(1-q)\ge q+q^2/2\) and \(\frac{q^2}{q+q^2/2} = \frac{r}{1+3r/2}\).
The second inequality is equivalent to
\[
 \lambda e^{-y}(y-3b/2)\le b.
\]
Since $e^{-1+y-3b/2}\geq y-3b/2$, the maximum of the left-hand side above is $\lambda e^{-1-3b/2}$, which is at most $\frac{b\lambda}{12}\leq b$ as $\lambda\leq 12$.

At exponent \(3/2\), the simpler estimate \(\log(1+r)\ge q\) gives
\[
 \frac{yq^{3/2}}{\log(1+r)}\le y\sqrt q \le y\sqrt r
 \le\sqrt{12}\,y e^{-y/2}\le \frac{2\sqrt{12}}e,
\]
where the last step follows from the fact that the maximum of $ye^{-y/2}$ occurs at $y=2$.

Interpolating these inequalities, we can choose \(p\in(3/2,2)\) 
close enough to \(2\) so that \eqref{eq:scalar-p} holds with $\rho = \left(\frac{2\sqrt{12}}e\right)^{4-2p}b^{2p-3}$ < 1.
\end{proof}

\vspace{-15pt}

\paragraph{The difference of means.}
By mild abuse of notation, for a vertex \(u\) below \(v\), 
let $[u, v)$ be the vertices along the path from \(v\) to \(u\), 
excluding \(v\) and including \(u\). Define $Q(u) = \prod_{j\in [u, v)} q_j$, where \(Q(v) = 1\). 

Let \(C_0=(1-\rho)^{-1}\) where $\rho$ is from Lemma \ref{lem:scalar-estimate}. We claim that \[
S(v):=\sum_{u\in V(T_v)} Q(u)^p \leq 1+C_0y_v.
\]
Indeed, since $S(v) = 1 + \sum_{i} q_i^p S(i)$ 
where the sum is over children \(i\) of $v$, 
we may use induction and \eqref{eq:scalar-p} to get
\[
 S(v) \leq 1+\sum_i q_i^p(1+C_0y_i)
 \le 1+(1+C_0\rho)\sum_i\log(1+r_i)
 =1+C_0y_v.
\]
Unrolling the first recurrence in \eqref{eq:moment-recursion} gives
\[
\delta_v = \sum_{u\in V(T_v)} (-1)^{|[u, v)|} Q(u).
\]
H\"older's inequality therefore gives
\begin{equation}\label{eq:delta-bound}
 |\delta_v|\le n_v^{\,1-1/p}(1+C_0y_v)^{1/p}.
\end{equation}

\vspace{-15pt}

\paragraph{The difference of variances.}
We next bound \(\gamma_v\). 
Set \(a=2-2/p\in(2/3,1)\) and \(u=1/(1-a)=p/(2-p)>p\). 
We also fix \(q_*=12/13\).
Since $y\geq 0$ and $\lambda\leq 12$, we have $0<r=\lambda e^{-y}\leq 12$ and hence $q=\frac{r}{1+r}\leq \frac{12}{13}=q_*$.
The inequalities $\log(1+r) \geq q$ and \eqref{eq:scalar-p} imply
\[
 \frac{q^u}{\log(1+r)}\le q^{u-1} \le q_*^{u-1},\qquad
 \frac{yq^u}{\log(1+r)}\leq q^{u-p}\rho \le q_*^{u-p}\rho<1.
\]
Choose \(G>0\) sufficiently large that
\[
 q_*^{u-1}+Gq_*^{u-p}\rho<G.
\]
and let $L$ denote the left-hand side above. Thus \(q^u(1+Gy)\le L \log(1+r)\). We can also choose constant $H > 0$ such that 
\[
\frac{q^u(1+C_0y)^{2u/p}}{\log(1+r)} 
\leq r^{u-1} (1+C_0y)^{2u/p} 
\leq 12^{u-1}e^{-(u-1)y}(1+C_0y)^{2u/p} \leq H
\]
for all $y\geq 0$ and $\lambda\in K$.

Choose \(C_1\) so that
\(C_1(G^{1/u}-L^{1/u})\ge H^{1/u}\).
We now show by induction that
\begin{equation}\label{eq:gamma-bound}
 |\gamma_v|\le C_1n_v^a(1+Gy_v)^{1/u}.
\end{equation}
The leaf case is immediate since $\gamma_v=0$. For the induction step, use
H\"older's inequality with conjugate exponents \(u\) and \(1/a\).
By induction hypothesis, the magnitude of the first sum in the recurrence for \(\gamma_v\) is bounded by
\begin{align*}
 \sum_{i} q_i|\gamma_i| &\leq C_1 \sum_i q_i n_i^a (1+Gy_i)^{1/u}  
 \leq C_1\left(\sum_iq_i^u(1+Gy_i)\right)^{1/u}
       \left(\sum_i n_i\right)^a \\
 &\leq C_1 \left(L\sum_i \log(1+r_i)\right)^{1/u}\left(\sum_i n_i\right)^a
 = C_1 L^{1/u} y_v^{1/u}\left(\sum_i n_i\right)^a.
\end{align*}
Using \eqref{eq:delta-bound}, $\delta_i^2 \leq n_i^{2-2/p} (1+C_0y_i)^{2/p}=n_i^a (1+C_0y_i)^{2/p}$. Since $1-q_i\leq 1$, it follows that the second sum $\sum_i q_i(1-q_i) \delta_i^2$ in the recurrence is at most
\[
 \sum_i q_i n_i^a(1+C_0y_i)^{2/p}
 \le H^{1/u}y_v^{1/u}\left(\sum_i n_i\right)^a.
\]
Their sum is bounded by the right side of
\eqref{eq:gamma-bound}, by the choice of \(C_1\) and
\(\sum_i n_i\le n_v\). This proves the claim.

Finally, note that \(q_v(1-q_v)\le r_v\le12e^{-y_v}\).
Multiplying the squares of \eqref{eq:delta-bound} and
\eqref{eq:gamma-bound} by this factor absorbs all the polynomial
factors in \(y_v\) into uniform constants, proving
\eqref{eq:root-moments}.
\end{proof}

\section{A uniform central limit theorem}\label{sec:clt}

Let \(\Phi\) denote the standard normal distribution function.

\begin{proposition}\label{prop:clt}
There are constants \(0<c<C<\infty\) such that
\begin{equation}\label{eq:linear-variance}
 cn\le\sigma^2\le Cn
\end{equation}
for every nonempty forest with \(n\) vertices and every \(\lambda\in K\).
Moreover, as $n$ approachs infinity,
\begin{equation}\label{eq:clt}
 \sup_{\substack{\emph{$n$-vertex forest }F\\ \lambda\in K}}
 \ \sup_{x\in\mathbb R}
 \left|\Prob_\lambda\!\left(\frac{X-\mu}{\sigma}\le x\right)
                   -\Phi(x)\right|\longrightarrow0.
\end{equation}
\end{proposition}

\begin{proof}
We prove the three claims in the statement in order.

\vspace{-15pt}

\paragraph{Lower variance bound.}
Fix a bipartition with vertex classes \(L,R\). Conditional on \(I\cap R\), the available vertices of \(L\) are independently
occupied with probability \(p=\lambda/(1+\lambda)\). Hence
\[
\operatorname{Var}(X\mid I\cap R)=p(1-p)\cdot \#\{v\in L:N(v)\cap I=\varnothing\}.
\] 
By total variance,
\[
\sigma^2
\ge
p(1-p)\sum_{v\in L}\mathbb{P}_\lambda(N(v)\cap I=\varnothing) \stackrel{(*)}{\geq} p(1-p)(1+\lambda)^{-d(v)},
\]
where $(*)$ follows from \((3.3)\).
The same inequality holds with \(L\) and \(R\) interchanged by symmetry,
so we can average the two inequalities. 
Since a forest has average degree at most \(2\),
Jensen's inequality gives
\begin{equation}\label{eq:variance-lower}
 \sigma^2\ge
 \frac{\lambda}{2(1+\lambda)^2}
       \sum_{v\in V(F)}(1+\lambda)^{-d(v)}
 \ge\frac{\lambda n}{2(1+\lambda)^4}\ge cn
\end{equation}
for some constant $c > 0$.

\vspace{-15pt}

\paragraph{Upper variance bound.}
For a tree, 
a \emph{centroid} is a vertex whose
deletion leaves components of at most half the original number of vertices. 
It is easy to verify that every tree has at least one centroid.

Fix an integer \(b\ge1\). Starting with a forest $F$, 
whenever a remaining component has more
than \(b\) vertices, choose a centroid of that component and
reveal whether it is occupied.
If the centroid is absent, we delete this vertex;
if the centroid is occupied, we record
its contribution \(1\) and delete its closed neighborhood.
Repeat until all components have size at most \(b\).
Every child component in this decomposition has at most half the number of vertices of its parent. 

Given any valid reveal history,
the unexposed set has the hard-core law on the
remaining forest, with the original fugacity. This follows
inductively from root conditioning. 
With a fixed tie-breaking rule, all choices of 
components and centroids are deterministic and
introduce no further conditions on that law.

Let \(\mathcal F_j\) be the information after \(j\) reveals, and let
\(M_j=\E(X\mid\mathcal F_j)\) and
\(S_j=\Var(X\mid\mathcal F_j)\). 
Write \(M,S\) for their respective terminal values. 
At a reveal in a component with \(n_j\) vertices, let
\(\xi_j\) be the centroid-occupation indicator, and let \(q_j,\delta_j,\gamma_j\)
be the root quantities in that component 
(from the previous section). The quantities
\(n_j,q_j,\delta_j,\gamma_j\) are \(\mathcal F_j\)-measurable,
and \(\E(\xi_j\mid\mathcal F_j)=q_j\).
The definition of $\delta_v$ and $\gamma_v$ together with two recurrences
in \eqref{eq:moment-recursion} give
\begin{equation}\label{eq:reveal-updates}
 \begin{split}
 M_{j+1}-M_j&=\delta_j(\xi_j-q_j),\\
 S_{j+1}-S_j&=\gamma_j(\xi_j-q_j)
                       -q_j(1-q_j)\delta_j^2.
 \end{split}
\end{equation}
To see the second identity, the pre-reveal variance of the affected
component is the average of its two conditional variances plus
\(q_j(1-q_j)\delta_j^2\). Its selected conditional variance differs
from their average by \(\gamma_j(\xi_j-q_j)\).
Other components remain independent and unchanged.

There are at most \(n\) reveals, so all martingale sums below have deterministic finite
length. The centered terms $M_{j+1}$ and $S_{j+1}$ in \eqref{eq:reveal-updates} have
second moments \(q_j(1-q_j)\delta_j^2\) and
\(q_j(1-q_j)\gamma_j^2\), conditioned on $M_j$ and $S_j$ respectively.

We now use the halving property to show that along every possible history, there is
\begin{equation}\label{eq:centroid-sums}
 \sum_j n_j^a\le\frac{n b^{a-1}}{1-2^{a-1}},
 \qquad
 \sum_j n_j^{2a}\le\frac{n^{2a}}{1-2^{-(2a-1)}},
\end{equation}
where $a\in(2/3,1)$ is from Proposition~\ref{prop:root-moments}.
Indeed, for the first bound, assign a charge \(n_j^{a-1}\) to each vertex
of the processed component. For a fixed vertex, the component sizes containing it halve until its last processed component,
whose size exceeds \(b\). Hence \[
\sum n_j^{a-1}
\le b^{a-1}\sum_{h\ge0}2^{h(a-1)}
=\frac{b^{a-1}}{1-2^{a-1}}.
\]
Summing over all vertices proves the bound. For the second, at depth
\(h\) the processed components are disjoint, their total sizes are
at most \(n\), and each has size at most \(n2^{-h}\).
Thus
\[
	\sum_{j:\,\mathrm{depth}(j)=h} n_j^{2a}
	\le (n2^{-h})^{2a-1}\sum_{j:\,\mathrm{depth}(j)=h}n_j
	\le n^{2a}2^{-(2a-1)h}.
\]
Summing over depths proves the claim,
regardless of whether the initial forest is connected or not.

We claim that
\begin{equation}\label{eq:terminal-stability}
\begin{split}
 \E(M-\mu)^2&= O(nb^{a-1}),\\
 \E|S-\sigma^2|& =O\bigl(n^a+nb^{a-1}\bigr).
 \end{split}
\end{equation}
To obtain the first bound, $$\E(M-\mu)^2=\E(\sum_j (M_{j+1}-M_j))^2=\sum_j \E(q_j(1-q_j)\delta_j^2)\stackrel{(*)}{\leq} C\sum_j n_j^a \stackrel{(**)}{=}O(nb^{a-1}),$$
where we apply Proposition \ref{prop:root-moments} for $(*)$ and (\ref{eq:centroid-sums}) for $(**)$. For the second bound, we use
\[\E\Big(\sum_j \gamma_j(\xi_j-q_j)\Big)^2 =\sum_j \E(q_j(1-q_j)\gamma_j^2) \stackrel{(\diamond)}{\leq} C \sum_j n_j^{2a}=O(n^{2a}),\]
where $(\diamond)$ follows from Proposition \ref{prop:root-moments}. Therefore Cauchy-Swartz implies that $\E|\sum_j \gamma_j(\xi_j-q_j)|=O(n^{a})$. The estimate then follows from $\E|S-\sigma^2|=\E|\sum_j \gamma_j(\xi_j-q_j)-q_j(1-q_j)\delta_j^2|$ and the first bound.

Choose \(b=1\). The terminal components are single vertices, so
\(S\le n/4\). Total variance and the first inequality in
\eqref{eq:terminal-stability} give
\(\sigma^2=\E S+\Var M\le n/4+O(n)\).
Together with \eqref{eq:variance-lower}, this shows that $\sigma^2 = \Theta(n)$.

\vspace{-15pt}

\paragraph{Convergence to normal.}
Now take \(b=\lceil n^{1/4}\rceil\) instead.
Conditional on the terminal information, \(X\) is a known integer
plus a sum of independent random variables \(X_i\), where \(X_i\)
counts the occupied vertices in a terminal component and
\(0\le X_i\le b\).
Therefore, $|X_i-\E X_i|\leq b$ and hence
\[
 \sum_i\E|X_i-\E X_i|^3\le b\sum_i \E|X_i-\E X_i|^2=b\sum_i\Var X_i=bS,
\]
where all moments in this display are under that conditional law.
On \(S\ge\sigma^2/2\), the Berry--Esseen inequality for independent but
not necessarily identically distributed variables gives error
at most \(O(bS/S^{3/2})=O(b/\sqrt S)\le O(b/\sqrt n)\) relative to the normal law
with mean \(M\) and variance \(S\).
We use only this standard one-dimensional inequality; an explicit
version follows, for example, from Rai\v{c}~\cite[Theorem~1.1,
\(d=1\)]{Raic}.

For \(r\ge1/2\) and \(h\in\mathbb R\), differentiation of the
normal distribution function in its shift and scale yields
\begin{equation}\label{eq:normal-comparison}
 \sup_x\left|\Phi\!\left(\frac{x-h}{\sqrt r}\right)-\Phi(x)\right|
 = O\bigl(|h|+|r-1|\bigr).
\end{equation}
Indeed, write $\Phi\left(\frac{x-h}{\sqrt r}\right)-\Phi(x)=\left(\Phi\left(\frac{x-h}{\sqrt r}\right)-\Phi(\frac{x}{\sqrt r})\right)+\left(\Phi(\frac{x}{\sqrt r})-\Phi(x)\right).$ Since $\Phi'(x)\leq \frac{1}{\sqrt{2\pi}}$,
\[\left|\Phi\left(\frac{x-h}{\sqrt r}\right)-\Phi(\frac{x}{\sqrt r})\right|\leq \left|\frac{h}{\sqrt r} \sup_x \Phi'(x)\right|=O(|h|).\] Next write $f(s)=\Phi(\frac{x}{\sqrt s})$, where $s\geq \frac{1}{2}$. The mean value theorem implies that
\[\left|\Phi(\frac{x}{\sqrt r})-\Phi(x)\right|=|f(r)-f(1)|\leq \sup_{s\geq 1/2} f'(s) \cdot|r-1|=O(|r-1|).\]
This proves (\ref{eq:normal-comparison}). 

Finally, apply \eqref{eq:normal-comparison} with
\(h=(M-\mu)/\sigma\), \(r=S/\sigma^2\); in particular, $\frac{X-\mu}{\sigma}=\sqrt{r}\cdot\frac{X-M}{S}+h$.
The complementary event $S<\sigma^2/2$ has probability 
$$\mathbb{P}\left(|S-\sigma^2|>\frac{\sigma^2}{2}\right)\leq \frac{2\E|S-\sigma^2|}{\sigma^2}=O(b^{a-1}+n^{a-1}),$$ by (\ref{eq:terminal-stability}) and the fact $\sigma^2=\Theta(n)$ we proved earlier. Putting everything together, the left side of
\eqref{eq:clt} is bounded by the sum of 1) the conditional Berry--Esseen estimate $O(b/\sqrt{n})$, 2) the mean shift $\E(|h|)=O(\frac{\sqrt{n}b^{a-1}}{\sqrt{n}})=O(b^{(a-1)/2})$ by \eqref{eq:terminal-stability}, and 3) the complementary event estimate $O(b^{a-1}+n^{a-1})$. That is,
\[
 O\left(\frac b{\sqrt n}
          +b^{(a-1)/2}+b^{a-1}+n^{a-1}\right).
\]
This tends to zero as $n\to\infty$ because \(a<1\), uniformly over all forests
and all fugacities in \(K\).
\end{proof}

\section{Log-concavity around mean}
\label{sec:fourier}

We now upgrade the central limit theorem from the previous section into convergence of the characteristic function, 
and show that the independence polynomial is log-concave around $\mu$ for sufficiently large forests.

The following estimate uses conditional independence, as in the
central-to-local limit method of~\cite{DT,JPSS}.

\begin{lemma}\label{lem:fourier}
There exists a constant $c > 0$ such that for every forest with \(n\) vertices, every \(\lambda\in K\), and every \(t\in \R\),
\begin{equation}\label{eq:fourier}
 \left|\E_\lambda e^{itX}\right|
 \le \exp\{-c n\sin^2(t/2)\}.
\end{equation}
\end{lemma}

\begin{proof}
First let \(G\) be any bipartite graph with classes \(L,R\),
and fix \(\lambda>0\). Given \(I\cap R\), let \(B\) be the number
of available vertices of \(L\). Then
\begin{equation}\label{eq:cond-modulus}
 \left|\E(e^{itX}\mid I\cap R)\right|
 =\left(1-\frac{4\lambda}{(1+\lambda)^2}
                          \sin^2(t/2)\right)^{B/2}.
\end{equation}
Set
\[
 \beta=\min\left\{\frac{\lambda}{1+\lambda},
                         \frac{2\lambda}{(1+\lambda)^2}\right\},
 \qquad h=1-\beta\sin^2(t/2),\qquad
 \tau=(1+\lambda)h-1.
\]
The inequality \(\sqrt{1-z}\le1-z/2\) for \(0\le z\le1\)
shows that the right-hand side of \eqref{eq:cond-modulus} is at most \(h^B\).
Our choice of \(\beta\) ensures \(0\le\tau\le\lambda\) and \(0 < h \le 1\).

Use \(Z_G(s,\lambda)\) for the partition function with fugacity
\(s\) on \(L\) and \(\lambda\) on \(R\),
which we can compute to be \[
Z_G(s,\lambda) = \sum_{J\subseteq R} \lambda^{|J|} (1+s)^B.
\]
which combined with \(h=(1+\tau)/(1+\lambda)\) gives
\begin{equation}\label{eq:inhomogeneous-ratio}
 \E_\lambda h^B = \frac{\sum_{J\subseteq R} \lambda^{|J|}(1+\lambda)^Bh^B}{Z_G(\lambda)}
 =\frac{Z_G(\tau,\lambda)}{Z_G(\lambda,\lambda)}.
\end{equation}

At these inhomogeneous fugacities, each vertex of \(R\) has
occupation probability at most \(\lambda/(1+\lambda)\), also
after conditioning some other vertices of \(R\) to be absent.
Thus, for \(v\in L\),
\[
 \Prob_{s,\lambda}(v\in I)
 =\frac{s}{1+s}\Prob_{s,\lambda}(N(v)\cap I=\varnothing)
 \ge\frac{s}{1+s}(1+\lambda)^{-d(v)}.
\]
For \(s>0\), logarithmic differentiation gives
\[
 \frac{d}{ds}\log Z_G(s,\lambda)
 =\frac1s\sum_{v\in L}\Prob_{s,\lambda}(v\in I)
 \ge\frac1{1+s}\sum_{v\in L}(1+\lambda)^{-d(v)}.
\]
Integrating the above inequality for $s$ from \(\tau\) to \(\lambda\) (using continuity if
\(\tau=0\)) gives
\[
\log\left(\frac{Z_G(\lambda, \lambda)}{Z_G(\tau, \lambda)}\right) \geq -\log(h)\sum_{v\in L}(1+\lambda)^{-d(v)},
\]
so applying \eqref{eq:inhomogeneous-ratio} gives
\[
 \left|\E_\lambda e^{itX}\right| \le \E_\lambda h^B
 \le h^{\,\sum_{v\in L}(1+\lambda)^{-d(v)}}
 \le
 \exp\left\{-\beta\sin^2(t/2)
                   \sum_{v\in L}(1+\lambda)^{-d(v)}\right\}.
\]
For a forest, we choose the bipartition $L\sqcup R$ so that \(|L|\ge n/2\).
Since \(\sum_{v\in L}d(v)=|E(F)|\le n\), Jensen's inequality yields
\[
 \sum_{v\in L}(1+\lambda)^{-d(v)}
 \ge |L|(1+\lambda)^{-|E(F)|/|L|}
 \ge\frac{n}{2(1+\lambda)^2} \ge \frac{n}{338}.
\]
So \eqref{eq:fourier} holds with $c = (\inf_{\lambda\in K}\beta)/338$.
\end{proof}

\begin{proposition}\label{prop:curvature}
Uniformly over forests with \(n\) vertices and fugacities
\(\lambda\in K\) for which \(\mu=k\) is an integer, as $n$ approaches infinity we have
\begin{equation}\label{eq:curvature-limit}
 \sigma^3\bigl(2\Prob_\lambda(X=k)
       -\Prob_\lambda(X=k-1)-\Prob_\lambda(X=k+1)\bigr)
 \longrightarrow\frac1{\sqrt{2\pi}}.
\end{equation}
Consequently, for all sufficiently large $n$ and every forest $F$ on $n$ vertices,
\begin{equation}\label{eq:mean-lc}
 i_k(F)^2>i_{k-1}(F)i_{k+1}(F)
\end{equation}
at every integer \(k\) between $\E_{1/4}X$ and $\E_{12}X$.
\end{proposition}

\begin{proof}
Let \(\chi(t)=\E e^{it(X-\mu)/\sigma}\).
Proposition~\ref{prop:clt} implies that \(\chi(t)\) tends to
\(e^{-t^2/2}\), uniformly on each fixed compact interval of \(t\)
and uniformly over the forests and fugacities.

For \(|t|\le\pi\sigma\), Lemma~\ref{lem:fourier} and
\(\sin z\ge2z/\pi\) on \(0\le z\le\pi/2\) give
\begin{equation}\label{eq:standardized-fourier}
 |\chi(t)|\le
 \exp\left\{-\frac{cn}{\pi^2\sigma^2}t^2\right\}
 \le e^{-c't^2},
\end{equation}
for some constants $c, c' > 0$, using the upper bound in \eqref{eq:linear-variance}.

Let \(p_j=\Prob_\lambda(X=j)\). If \(\mu=k\in\mathbb Z\), lattice Fourier inversion gives
\begin{align*}
 \sigma^3(2p_k-p_{k-1}-p_{k+1})
 &=\frac{\sigma^3}{2\pi}\int_{-\pi\sigma}^{\pi\sigma}
      (2-e^{it/\sigma}-e^{-it/\sigma})\chi(t)\,d(t/\sigma) \\
 &=\frac1{2\pi}\int_{-\pi\sigma}^{\pi\sigma}
      2\sigma^2\bigl(1-\cos(t/\sigma)\bigr)\chi(t)\,dt.
\end{align*}
The multiplier on $\chi(t)$ is at most \(t^2\) in absolute value and tends to
\(t^2\) uniformly on compact intervals, since
\(\sigma^2\ge cn\) uniformly approaches infinity. Equation~\eqref{eq:standardized-fourier}
provides the integrable dominating function \(t^2e^{-c't^2}\).
Thus the right side tends uniformly to
\[
 \frac1{2\pi}\int_{\mathbb R}t^2e^{-t^2/2}\,dt
 =\frac1{\sqrt{2\pi}},
\]
which proves \eqref{eq:curvature-limit}.

For sufficiently large \(n\), it follows that
\(2p_k>p_{k-1}+p_{k+1}\). The arithmetic--geometric mean
inequality implies \(p_k^2>p_{k-1}p_{k+1}\).
From \eqref{eq:measure},
\[
 p_k^2-p_{k-1}p_{k+1}
 =\frac{\lambda^{2k}}{Z_F(\lambda)^2}
       \bigl(i_k(F)^2-i_{k-1}(F)i_{k+1}(F)\bigr),
\]
so $i_k^2 > i_{k-1}i_{k+1}$ as well.

Finally, \eqref{eq:derivatives} shows that the mean is continuous
and strictly increasing in \(\lambda\). Every integer between
the endpoint means is therefore the mean for some \(\lambda\in K\).
These integers are interior support indices, because every
finite positive fugacity has mean strictly between \(0\) and
\(\alpha(F)\). This proves \eqref{eq:mean-lc}.
\end{proof}

\section{The range of the mean}\label{sec:mean}

The previous section shows that the independence polynomial is 
log-concave at all indices that can be the mean for some fugacity $\lambda \in K$. 
It now remains to bound this range of mean in order to prove Theorem~\ref{thm:central}.

\begin{proposition}\label{prop:mean}
For every nonempty forest \(F\) with \(n\) vertices,
\begin{equation}\label{eq:mean-range}
 \E_{1/4}X\le\frac n5,\qquad
 \E_{12}X>\frac{17\alpha(F)}{25}.
\end{equation}
\end{proposition}

\begin{proof}
At $\lambda = 1/4$, each vertex is occupied with probability at most 
\(\lambda/(1+\lambda)  = 1/5\), which gives the first bound.

For the second bound, fix \(F\) and let
\(Q(\lambda)=\log Z_F(\lambda)/\E_\lambda X\).
We first bound \(Q(1)\), then use a variance inequality to control
its growth with \(\lambda\).

\vspace{-15pt}

\paragraph{Partition function and mean at fugacity \(1\).}
Root each component, and use the descendant odds \(r_v\) at
fugacity \(\lambda = 1\). We factor
\[
 Z_{T_v}(1)=(1+r_v)\prod_i Z_{T_i}(1),
\]
where \(T_i\) are the child trees. Iterating it yields
\[
 \log Z_F(1)=\sum_v\log(1+r_v),
\]
where we note that $0<r_v\le1$ for all $v$.
If \(v\) is not the root of a component, let \(s_v\) be the occupation
odds of its parent, rooted in the component on the parent's side
after deleting the connecting edge between $v$ and the parent; 
again \(0<s_v\le1\).
Divide the weights by the product of the two root-absent partition
functions. The three allowed occupation states of \(v\) and its parent,
\((0,0),(1,0),(0,1)\), then have weights \(1,r_v,s_v\) respectively.
Thus the occupation probability of \(v\) in the entire forest is
\[
 \Prob_1(v\in I)=\frac{r_v}{1+r_v+s_v}
 \ge\frac{r_v}{2+r_v}.
\]
For a component root the same inequality holds with \(s_v=0\).
The function
\[
 r\longmapsto\frac{(2+r)\log(1+r)}r
\]
is increasing for \(0<r\le1\), and its value at \(r=1\) is \(3\log2\).
Consequently,
\[
 \log(1+r_v)\le3\log2\cdot\frac{r_v}{2+r_v}
             \le3\log2\cdot\Prob_1(v\in I).
\]
Summation gives
\begin{equation}\label{eq:logz-one}
 \log Z_F(1)\le3\log2\cdot\E_1X \, \Rightarrow\, Q(1)\leq 3\log2.
\end{equation}

\vspace{-15pt}

\paragraph{Bounding the growth of $Q$.}
For any bipartite graph and any \(\lambda>0\), the same
conditional-variance argument used for \eqref{eq:variance-lower} gives
\begin{equation}\label{eq:variance-mean}
 \sigma^2\ge\frac{\mu}{2(1+\lambda)}.
\end{equation}
In detail, conditioning on \(I\cap R\) gives conditional variance
\(\lambda/(1+\lambda)^2\) times the number of available vertices
in \(L\), whose expected occupation count is
\(\lambda/(1+\lambda)\) times that number. Hence
\(\sigma^2\ge\E|I\cap L|/(1+\lambda)\).
Averaging with the analogous inequality for \(R\) gives \eqref{eq:variance-mean}.

Equations~\eqref{eq:derivatives} and
\eqref{eq:variance-mean} imply
\[
 DQ(\lambda)=1-\frac{\sigma^2}{\mu}Q(\lambda)
 \le1-\frac{Q(\lambda)}{2(1+\lambda)}.
\]
Using \(\sqrt{\lambda/(1+\lambda)}\) as the integrating factor, this gives
\[
 \frac{d}{d\lambda}\left(
    \sqrt{\frac{\lambda}{1+\lambda}}\,Q(\lambda)\right)
 \le\frac1{\sqrt{\lambda(1+\lambda)}}.
\]
Integrating both sides from \(1\) to \(\lambda\ge1\), and using
\eqref{eq:logz-one}, gives
\begin{equation}\label{eq:q-bound}
 Q(\lambda)
 \le \sqrt{\frac{1+\lambda}{\lambda}}\left(\frac{3\log2}{\sqrt2}
       +2\log\frac{\sqrt\lambda+\sqrt{1+\lambda}}{1+\sqrt2}\right).
\end{equation}

Let $Q^*(\lambda)$ denote the upper bound on the right-hand side above. 
On the other hand, every subset of a fixed maximum independent set is independent,
so \(Z_F(\lambda)\ge (1+\lambda)^{\alpha(F)}\) for all $\lambda > 0$. Letting $\lambda = 12$, we get
\[
 \E_{12}X=\frac{\log Z_F(12)}{Q(12)}
 \geq\frac{\log(13)\alpha(F)}{Q^*(12)}
 \approx 0.6808\alpha(F) > \frac{17\alpha(F)}{25}.
\]
\end{proof}

\begin{proof}[Proof of Theorem~\ref{thm:central}]
Every integer in the stated interval lies between the endpoint
means by Proposition~\ref{prop:mean}.
Proposition~\ref{prop:curvature} therefore applies. Its size
threshold is uniform over the forest and fugacity.
\end{proof}

\section{Monotonicity of initial and final segments}\label{sec:ends}

For an independent set \(J\), write
\(e(J)=|V(F)\setminus N[J]|\) for its number of one-vertex
extensions. The standard double-counting identity
\begin{equation}\label{eq:extensions}
 (k+1)i_{k+1}(F)=
      \sum_{\substack{J\in\mathcal I(F),\, |J|=k}}e(J)
\end{equation}
counts pairs of nested independent sets of sizes \(k,k+1\).
This is the extension method used in~\cite[Section~2.2]{BG}.

We first show that an independent set of size $k$ has at most $2k$ incident edges on average.

\begin{lemma}\label{lem:degree}
Let \(T\) be a tree rooted at \(r\). The polynomial
\begin{equation}\label{eq:degree-polynomial}
 D_T:=\sum_{J\in\mathcal I(T)}
 \left(2|J|-\sum_{v\in J}d_T(v)-2\cdot\ind_{\{r\in J\}}\right)x^{|J|}
\end{equation}
has nonnegative coefficients.
Consequently, for \(0\le k\le\alpha(F)\) and a uniformly chosen
independent set \(J\) of size $k$ of any forest \(F\),
\begin{equation}\label{eq:degree-average}
 \E\sum_{v\in J}d_F(v)\le2k.
\end{equation}
\end{lemma}

\begin{proof}
We prove coefficientwise nonnegativity of two polynomials
simultaneously. Define
\[
 E_T:=
 \sum_{J\in\mathcal I(T),\, r\notin J}
 \left(2|J|-\sum_{v\in J}d_T(v)-1\right)x^{|J|}
 +Z_{T-N[r]}(x).
\]
This records the companion inequality when the
root is absent; its correction term $Z_{T-N[r]}(x)$ cancels out the contribution
of an occupied root. 
Note that for a single-vertex tree both polynomials are zero.

Suppose the root has child trees \(T_1,\ldots,T_m\), with roots
\(r_1,\ldots,r_m\). 
We claim that the following two recurrences hold:
\begin{align}
 D_T={}&\sum_iD_{T_i}\prod_{j\ne i}Z_{T_j}
       +\sum_i(Z_{T_i}-Z_{T_i-r_i})
                    \left(\prod_{j\ne i}Z_{T_j}-\prod_{j\ne i}Z_{T_j-r_j}\right)
       +x\sum_iE_{T_i}\prod_{j\ne i}Z_{T_j-r_j}, \label{eq:d-recursion}\\
 E_T={}&\sum_iD_{T_i}\prod_{j\ne i}Z_{T_j}
       +\sum_{\substack{S\subseteq\{1,\ldots,m\}\\|S|\ge2}}
          (|S|-1)\prod_{i\in S}(Z_{T_i}-Z_{T_i-r_i})\prod_{j\notin S}Z_{T_j-r_j} .
        \label{eq:e-recursion}
\end{align}

We first verify \eqref{eq:d-recursion}. When \(r\notin J\), the total weight
\(2|J|-\sum_{v\in J}d_T(v)\) has generating polynomial
\begin{align}
 \sum_{J\in\mathcal I(T),\, r\notin J}
 \left(2|J|-\sum_{v\in J}d_T(v)\right)x^{|J|} &= 
 \sum_{J\in\mathcal I(T),\, r\notin J} \sum_i \left(2|J_i| - \sum_{v\in J_i} d_{T_i}(v) - \ind_{r_i\in J_i}\right) x^{|J|}
 \\
 &=\sum_i\bigl(D_{T_i}+Z_{T_i}-Z_{T_i-r_i}\bigr)\prod_{j\ne i}Z_{T_j}. \label{eq:absent-part}
\end{align}
where $J_i = J\cap V(T_i)$. Indeed, the degree of \(r_i\) in \(T\) is its degree in \(T_i\)
plus \(1\); this leaves one, rather than two, copies of its
occupation indicator in the first line above.

When \(r\in J\), all \(r_i\) are absent. The root's contribution
to the weight defining \(D_T\) is \(-m\), which we can distribute to be $-1$ for each of the $m$ subtrees.
By the definition of \(E_{T_i}\),
\[
 \sum_{J_i\in\mathcal I(T_i),\, r_i\notin J_i}
 \left(2|J_i|-\sum_{v\in J_i}d_{T_i}(v)-1\right)x^{|J_i|}
 =E_{T_i}-Z_{T_i-N[r_i]}.
\]
Since \(xZ_{T_i-N[r_i]}=Z_{T_i}-Z_{T_i-r_i}\), the occupied-root
contribution to \(D_T\) is exactly
\begin{align*}
 \sum_{J\in\mathcal I(T),\, r\in J}
 \left(2|J|-\sum_{v\in J}d_T(v) -2\right)x^{|J|}
 &= x \sum_{J\in\mathcal I(T-N[r])}\sum_i \left(2|J_i| - \sum_{v\in J_i} d_{T_i}(v) - 1\right) x^{|J|}
 \\
 &= x \sum_{i} \sum_{J_i\in\mathcal I(T_i),\, r_i\notin J_i}\left(2|J_i| - \sum_{v\in J_i} d_{T_i}(v) - 1\right) x^{|J_i|} \prod_{j\neq i} Z_{T_j-r_j}
 \\
 &= x\sum_iE_{T_i}\prod_{j\ne i}Z_{T_j-r_j}
       -\sum_i(Z_{T_i}-Z_{T_i-r_i})\prod_{j\ne i}Z_{T_j-r_j}.
\end{align*}
Adding the two contributions gives the desired recurrence for $D_T$.

For \(E_T\), we reuse \eqref{eq:absent-part} to obtain 
\begin{align*}
    E_T &= \left(\sum_i\bigl(D_{T_i}+Z_{T_i}-Z_{T_i-r_i}\bigr)\prod_{j\ne i}Z_{T_j}\right) - Z_{T-r} + Z_{T-N[r]} \\
    &= \sum_i D_{T_i} \prod_{j\ne i}Z_{T_j} + \left(\sum_i\bigl(Z_{T_i}-Z_{T_i-r_i}\bigr)\prod_{j\ne i}(Z_{T_i-r_i}+(Z_{T_i}-Z_{T_i-r_i}))\right) \\
    &\hspace{4cm} - \left(\prod_i(Z_{T_i-r_i}+(Z_{T_i}-Z_{T_i-r_i})) - \prod_iZ_{T_i-r_i}\right).
\end{align*}
After expanding, for each nonempty set \(S\), $\prod_{i\in S}(Z_{T_i}-Z_{T_i-r_i})\prod_{j\notin S}Z_{T_j-r_j}$ is counted
\(|S|\) times in
\(\sum_i(Z_{T_i}-Z_{T_i-r_i})\prod_{j\ne i}Z_{T_j}\), and once in
\(\prod_iZ_{T_i}-\prod_iZ_{T_i-r_i}\). The remaining coefficient is
\(|S|-1\), proving \eqref{eq:e-recursion}.

All terms on the right sides of \eqref{eq:d-recursion} and
\eqref{eq:e-recursion} have nonnegative coefficients if 
all $D_{T_i}$ and $E_{T_i}$ do: in particular
\(\prod_{j\ne i}Z_{T_j}-\prod_{j\ne i}Z_{T_j-r_j}\) does, since
each \(Z_{T_j}-Z_{T_j-r_j}\) does. Simultaneous induction proves nonnegativity of $D_T$.

For a forest $F$, choose one root $r_i$ in each component $T_i$. Consider the sum 
\(\sum_{i} D_{T_i} \prod_{j\neq i} Z_{T_j},\) 
which clearly has nonnegative coefficients. The coefficient of
\(x^k\) in this sum is
\[
 \sum_{J\in\mathcal I(F),\,|J|=k}
 \left(2k-\sum_{v\in J}d_F(v)
        -2\sum_{i} \ind_{\{r_i\in J\}}\right)\ge0.
\]
In particular, the sum of $2k - \sum_{v\in J} d_F(v)$ is nonnegative, which proves \eqref{eq:degree-average}.
\end{proof}

\begin{proposition}\label{prop:ends}
Every nonempty forest \(F\) with \(n\) vertices satisfies
\begin{equation}\label{eq:prefix}
 i_0(F)\le i_1(F)\le\cdots\le i_{\lceil n/4\rceil}(F).
\end{equation}
Every bipartite graph \(F\) satisfies
\begin{equation}\label{eq:tail}
 i_{\lceil(2\alpha-1)/3\rceil}(F)\ge
 i_{\lceil(2\alpha-1)/3\rceil+1}(F)\ge\cdots\ge i_\alpha(F),
 \qquad \alpha=\alpha(F).
\end{equation}
\end{proposition}

The second assertion above is the bipartite case of the decreasing-tail
theorem of Levit and Mandrescu~\cite{LM}; the proof below uses
only extension counting.

\begin{proof}
For a uniform independent set $J$ of size $k$, \(J\cap N(J)=\varnothing\)
and \(|N(J)|\le\sum_{v\in J}d_F(v)\). Lemma~\ref{lem:degree}
and \eqref{eq:extensions} give
\[
 \frac{(k+1)i_{k+1}(F)}{i_k(F)}
 =\E e(J)\ge n-k-\E\sum_{v\in J}d_F(v)\ge n-3k.
\]
If \(k\le\lfloor(n-1)/4\rfloor\), then \(n-3k\ge k+1\).
This proves \eqref{eq:prefix} since
\(\lfloor(n-1)/4\rfloor+1=\lceil n/4\rceil\).

For the tail, the bipartite graph $F'$ induced by
\(V(F)\setminus N[J]\) has
\(\alpha(F')\leq \alpha-k\): any independent set in it can be adjoined to \(J\).
On the other hand, a bipartite graph on \(e(J)\) vertices has an independent set
of size at least \(e(J)/2\), so \(e(J)\le2(\alpha-k)\).
Using \eqref{eq:extensions} again gives
\[
 (k+1)i_{k+1}(F)\le2(\alpha-k)i_k(F).
\]
For \(k\ge\lceil(2\alpha-1)/3\rceil\), we have
\(2(\alpha-k) \leq k+1\), which proves \eqref{eq:tail}.
\end{proof}

\begin{proof}[Proof of Theorem~\ref{thm:main}]
Take $n$ sufficiently large that
Theorem~\ref{thm:central} applies. The independence sequence is
nondecreasing through $\lceil n/4\rceil$, strictly log-concave
at indices from $\lceil n/5\rceil$ to $\lfloor17\alpha(F)/25\rfloor$,
and nonincreasing from $\lceil(2\alpha(F)-1)/3\rceil$ onward.

Since $\alpha(F)\ge n/2$ and $17/25>2/3$, for all sufficiently large $n$
these ranges overlap:
\[
\left\lceil\frac n5\right\rceil
\le \left\lceil\frac n4\right\rceil
< \left\lceil\frac{2\alpha(F)-1}{3}\right\rceil
\le \left\lfloor\frac{17\alpha(F)}{25}\right\rfloor.
\]
Log-concavity makes the successive coefficient ratios nonincreasing
across the middle range. Together with the monotonicity of the two
ends, this prevents a decrease followed by an increase.
Hence the entire independence sequence is unimodal.
\end{proof}

\section{Concluding remarks}\label{sec:Conclude}
We emphasize that the original conjecture for trees and forests of all sizes remains open. 
We end with two additional questions, concerning the location of the
peak and the range of log-concavity.

\vspace{-15pt}

\paragraph{Peak interval.} 
The proof of Theorem~\ref{thm:main} gives bounds on the location of the peak(s) in the independence polynomial $Z_F(x)$ for all sufficiently large forests: define
\[a_p=\lim_{n\rightarrow \infty} \inf_{\text{$n$-vertex forest }F} \left\{ \frac{k}{\alpha(F)}:\ i_{k-1}(F) \leq i_k(F) \geq i_{k+1}(F) \right\},
\] 
\[b_p=\lim_{n\rightarrow \infty} \sup_{\text{$n$-vertex forest }F} \left\{ \frac{k}{\alpha(F)}:\ i_{k-1}(F) \leq i_k(F) \geq i_{k+1}(F) \right\}.
\] 
Proposition~\ref{prop:ends} shows $a_p\ge 1/4$ and $b_p\le 2/3$. Moreover, the well-covered spider graphs from~\cite{LM03} has their peaks at $2\alpha/3 - O(1)$, so $b_p = 2/3$ exactly. It still remains to determine the value of $a_p$. Does the analog of $a_p$ for trees has the same value as for forests? 

\vspace{-15pt}

\paragraph{Log-concavity interval.} 
Define $[a_{lc},b_{lc}] \subseteq [0,1]$ to be the largest interval such that for every sufficiently large forest $F$, 
the independence polynomial $Z_F(x)$ has a log-concave sequence of coefficients between indices $(a_{lc}+o(1))\alpha(F)$ and $(b_{lc}-o(1))\alpha(F)$. 

As remarked at the end of Section~\ref{sec:hardcore}, 
one can replace the lower bound $1/4$ of fugacity $\lambda$ with any sufficiently small positive constant: 
the only places where a positive lower bound was needed is for the constants $c$ in 
Proposition~\ref{prop:clt} and Lemma~\ref{lem:fourier}. 
Therefore, our proof of Theorem~\ref{thm:central} goes through with $k\geq \varepsilon n$ 
for any fixed $\varepsilon > 0$ (instead of $k\geq  n/5$). Hence $a_{lc} = 0$ and $b_{lc} \geq 17/25$. 

As mentioned previously, the constant $17/25$ is not optimal, and the exact value of $b_{lc}$ is not known. For all examples we know of, the failures to log-concavity appear only at degree $(1-o(1))\alpha(F)$, 
so it is possible that $b_{lc}=1$.


\printbibliography

\end{document}